\documentclass{amsart}

\usepackage{amsmath,amssymb,amsthm,mathrsfs}

\newtheorem{prop}{Proposition}[section]

\newtheorem{lem}[prop]{Lemma}

\newtheorem*{thm2}{Theorem}

\newtheorem{cor3}{Corollary}

\theoremstyle{definition}

\newtheorem{rem}[prop]{Remark}

\newtheorem*{ack}{Acknowledgement}

\newcommand{\AAA}{\mathcal A}

\newcommand{\BB}{\mathcal B}

\newcommand{\CC}{\mathcal C}

\newcommand{\rmd}{\mathrm d}

\newcommand{\DD}{\mathcal D}

\newcommand{\EE}{\mathcal E}
\newcommand{\rme}{\mathrm e}

\newcommand{\LLL}{\mathscr L}

\newcommand{\MM}{\mathcal M}

\newcommand{\NN}{\mathcal N}

\newcommand{\N}{\mathbb N}

\newcommand{\R}{\mathbb R}

\newcommand{\UU}{\mathcal U}

\newcommand{\WW}{\mathcal W}

\newcommand{\Z}{\mathbb Z}

\newcommand{\lra}{\longrightarrow}
\newcommand{\ra}{\rightarrow}

\DeclareMathOperator{\dist}{\mathrm{dist}}

\DeclareMathOperator{\ev}{\mathrm{ev}}

\DeclareMathOperator{\grad}{\mathrm{grad}}

\DeclareMathOperator{\Hess}{\mathrm{Hess}}

\DeclareMathOperator{\ind}{\mathrm{ind}}
\DeclareMathOperator{\inj}{\mathrm{inj}}

\DeclareMathOperator{\length}{\mathrm{length}}
\DeclareMathOperator{\loc}{\mathrm{loc}}

\begin{document}

\author{Stefan Suhr}
\address{DMA, \'Ecole Normale sup\'erieure and Universit\'e Paris-Dauphine,
45 rue d'Ulm, 75230 Paris Cedex 05, France}
\email{stefan.suhr@ens.fr}
\author{Kai Zehmisch}
\address{Mathematisches Institut, Westf\"alische Wilhelms-Universit\"at M\"unster, Einsteinstr. 62,
D-48149 M\"unster, Germany}
\email{kai.zehmisch@wwu.de}

\title[Linking and closed orbits]{Linking and closed orbits}

\date{18.12.2015}

\begin{abstract}
  We show that the Lagrangian of classical mechanics
  on a Riemannian manifold
  of bounded geometry
  carries a periodic solution of motion
  with prescribed energy,
  provided the potential satisfies an asymptotic growth
  condition,
  changes sign,
  and the negative set of the potential
  is non-trivial
  in the relative homology.
\end{abstract}

\thanks{
  This research was supported through the programme
  {\it Research in Pairs} by the {\it Mathematisches Forschungsinstitut Oberwolfach}
  in February 2013.
  S. Suhr: The research leading to these results has received funding from the
  European Research Council under the European Union's Seventh Framework
  Programme (FP/2007-2013) / ERC Grant Agreement 307062.
  K. Zehmisch is partially supported by DFG grant ZE 992/1-1.  
  }

\maketitle


\section{Introduction\label{intro}}

In Hamiltonian mechanics one considers
integral curves $u=u(t)$
of the Hamiltonian vector field $X_H$
for the {\bf Hamiltonian}
\[
H(u)=\frac12 |u|^2+V\big(\pi(u)\big)
\]
on the phase space $T^*Q$.
The motion of particles
in the configuration space $Q$
is described by $q(t)=\pi\big(u(t)\big)$
for the projection $\pi$
of the cotangent bundle.
The kinetic energy $\frac12 |\,.\,|^2$
is defined via the (dual) norm of a Riemannian metric
$g$ on $Q$.
The potential energy $V$
is a smooth function on $Q$.
For the Liouville $1$-form $\lambda$
on $T^*Q$ the Hamiltonian vector field $X_H$
is determined by $i_{X_H}\rmd\lambda=-\rmd H$,
i.e.\ the trajectories locally
solve the Hamilton equations.

The total energy $H(u)$ along a solution
$u=u(t)$ is preserved.
So in order to understand
the dynamics of the Hamiltonian vector field
one can restrict to hypersurfaces
$M=\{H=E\}$ of fixed energy $E$.
An important question is the existence
of periodic solutions $u=u(t)$ on a regular
{\bf energy surface} $M$.
If the energy $E$ is greater than $\sup_QV$
the Hamiltonian flow on $M$
is conjugate to the geodesic flow
of the Jacobi metric $(E-V)^{-1}g$
according to the Euler-Maupertuis-Jacobi principle.
If additionally $Q$ is compact
the existence of periodic orbits follows from
the existence of closed geodesics
for the Jacobi metric,
see \cite[Section 4.4]{hoze94}.
Existence of closed geodesics
on compact Riemannian manifolds is
well known,
cf.\ \cite{{kli78},bang85}.
In the case that the energy surface $M$
intersects the zero section $Q$ of $T^*Q$
the Jacobi metric becomes singular along $\{V=E\}$.
As explained in \cite[Section 4.4]{hoze94}
periodic orbits on compact energy surfaces $M$
can be found via so-called brake orbits
of a perturbed Jacobi metric,
see \cite{bolo78,bolkoz78,glzi83,benci84}.

Alternative approaches to periodic
orbits on a compact energy surfaces
via symplectic
geometry are given in
\cite{vit87,hovi88,hoze94,vit99,frauschl07,geizeh12}.
Rabinowitz applied the minmax method
to the Lagrangian multiplier functional for the
symplectic action $\int\rmd\lambda$
on $1$-periodic curves $u=u(t)$
constraint by $\int H\big(u(t)\big)\rmd t=0$,
see \cite{rabi79},
which led to the Rabinowitz-Floer theory
\cite{cielfrau09}.

Periodic orbits do not always exist
if the energy surfaces $M$
is non-compact.
For example
the Hamiltonians
$\frac12p^2-q$ and
$\frac12p^2-\frac12q^2$ on $T^*\R$
do not allow any closed orbit.
Moreover,
if the potential equals $-E<0$
periodic orbits correspond
precisely to closed geodesics via $\pi$,
cf.\ \cite{geig08}.
On the other hand
if $\{V<E\}$
is connected with disconnected boundary
and the potential satisfies
asymptotic growths conditions
periodic solutions are obtained in \cite{off87}
using the Dirichlet principle
for the energy functional
of the Jacobi metric.

As it was shown in \cite{bpv09}
in order to obtain periodic orbits
the Lagrangian action
functional can be used directly.
For non-compact energy surfaces $M$
periodic solutions do exists
if the potential function on $\R^n$
satisfies certain
topological and asymptotic conditions.
In \cite{bprv13} the authors generalized
the existence result
to Riemannian manifolds with flat ends
that satisfy a vanishing condition
on the free loop space homology.
The aim of these notes is to generalize the results
obtained in \cite{bpv09,bprv13}
to Riemannian manifolds of bounded geometry
without restrictions on the topology of the loop space.
As we will discuss in Section \ref{reeb}
all the above mentioned classical
existence results
with a sign changing potential are
consequences of our theorem.

\subsection{The theorem}

We consider a connected manifold $Q$ of dimension $n$
together with a Riemannian
metric $g=\langle\,.\,,\,.\,\rangle$
of {\bf bounded geometry} in the following sense:
We require that the {\bf injectivity radius}
\[
\inj g>0,\label{inj}\tag{INJ}
\]
which by definition is the infimum of the injectivity
radii of all points of $Q$.
In particular,
the Riemannian manifold
$(Q,g)$
is complete,
i.e.
the geodesic flow is global.
Moreover,
we require that there exists a positive constant $C$
such that the
Ricci tensor $\mathrm{Ric}$
satisfies
\[
|\mathrm{Ric}|<C, \label{cb}\tag{CB}
\]
where $|\,.\,|$ denotes the norm
induced by the metric.

We replace the {\bf potential} $V$
by $V-E$ and require that
\[
0
\,\,\text{is a regular value of}\,\,
V
\,\,\text{and}\,\,
\{V=0\}
\,\,\text{is not empty.}\,\,
\label{reg}\tag{REG}
\]
This
is equivalent to the requirement
that $0$ is a regular value of $H$
and that the energy surface
$M=\{H=0\}$ intersects the zero section non-trivially.
We call $N=\{V<0\}$ the {\bf negative set}.
Notice, that the closure equals $\pi(M)$.
We require that
the relative homology group
of the negative set 
\[
H_*(N,\partial N)\neq 0
\label{lnk}\tag{LNK}
\]
is non-trivial
for some degree $*=1,\ldots,n$.

In addition,
the potential is required to satisfy the following
{\bf asymptotic conditions}:
We denote by $\grad V$ the gradient of the potential $V$,
which is the dual vector field of the differential $TV$
w.r.t.\ the metric,
and by $\nabla$ the covariant derivative w.r.t.\ the
Levi-Civita connection.
The second covariant derivative $\nabla TV$,
the so-called {\bf Hessian},
is a
symmetric bilinear form
which we denote by $\Hess V$.
We require that there exist a positive constant $K$
and a compact subset
$\hat{Q}$ of $Q$ such that
\[
|\grad V|\geq\frac1K
\qquad\text{on}\;\;Q\setminus\hat{Q}
\label{ac1a}\tag{AC$_1^a$}
\]
and
\[
\frac{|\Hess V|}{|\grad V|}\leq K
\qquad\text{on}\;\;Q\setminus\hat{Q}.
\label{ac1b}\tag{AC$_1^b$}
\]
To formulate the second asymptotic condition
we distinguish a base point $o$ in $\hat{Q}$
and denote by $\dist(o,q)$ the distance
to $q\in Q$,
i.e. the minimal {\bf length}
\[
\length(c)=\int\!|\dot c|\rmd t
\]
of piecewise $C^1$-curves $c$
connecting $o$ with $q$.
We require that
\[
\frac{|\Hess_qV|}{|\grad_qV|}\lra 0
\qquad\text{if}\;\;\dist(o,q)\ra\infty.
\label{ac2}\tag{AC$_2$}
\]
Of course \eqref{ac2} implies \eqref{ac1b}.
If $Q$ is compact the
asymptotic conditions are automatically satisfied.

Assuming the above requirements
the aim of these notes is to find
periodic critical points $q=q(t)$
of the action
with the {\bf Lagrangian}
\[
L(v)=\frac12 |v|^2+U\big(\pi(v)\big)
\]
constraint to the energy condition
$\frac12|\dot q|^2+V(q)=0$.
Here we set $U=-V$ and
$\pi$ denotes the projection of the tangent bundle of $Q$.
Notice, that the period of $q=q(t)$
is free to vary
but in view of \eqref{reg} it has to be non-zero.
The Euler-Lagrange equation for the critical points
are given by
$
\ddot q=\grad_qU
$,
where $\ddot q=\nabla_{\!\dot q\,}\dot q$
is the covariant derivative along the curve $q=q(t)$,
see Section \ref{critpoint}.

\begin{thm2}
  \label{mthm}
  Under the assumptions
  \eqref{inj} -- \eqref{ac2}
  the equation of motion
  \[
  \ddot q=\grad_qU
  \]
  has a contractible periodic solution $q=q(t)$
  such that $\frac12|\dot q|^2=U(q)$.
\end{thm2}

Solutions are contained in the closure of the negative set
and, moreover,
are in one to one correspondence
with closed integral curves $u=u(t)$
of the Hamiltonian
vector field $X_H$ on the energy surface $M$
via the Legendre transformation.
The correspondences is given by $q(t)=\pi\big(u(t)\big)$,
cf.\ \cite{abbschw06}.
If $Q$ is a line or a circle periodic solutions
exists if and only if the negative set has compact
closure.
In this case
the Hamiltonian lift $u=u(t)$
of $q=q(t)$ is always non-contractible.

\begin{rem}
  If $\pi_1(N)$ injects into $\pi_1(Q)$
  and if the dimension of $Q$ is greater than one
  the contractibility of $q=q(t)$
  implies the contractibility of $u=u(t)$ in $M$.
  Indeed,
  a perturbation of $u=u(t)$ can be brought
  into general position to be disjoint from
  $\partial N\subset M$.
  The projection $q=q(t)$
  is contained in a contractible
  neighbourhood over which the
  unit cotangent bundle $ST^*Q\simeq M$
  is trivial so that $u=u(t)$
  is homotopic to $\big(q_*,p(t)\big)$
  for a point $q_*$ in $Q$.
  Connecting $q_*$ with the boundary of $N$
  shows that the fibre over $q_*$ bounds
  a $n$-disc inside $M$,
  which can be used to contract $u=u(t)$.
\end{rem}

In order to  prove the theorem
we consider the Lagrangian action
\[
\int_0^TL\big(\dot q(t)\big)\rmd t
\]
for positive real numbers $T$
and $T$-periodic curves $q=q(t)$ in $Q$.
Critical points are the solutions
claimed in the theorem.
Following \cite{bpv09} we reparametrize the curves
by setting
\[
x(t)=q\big(\rme^{\tau}t\big)
\]
with $T=\rme^{\tau}$.
This results in the functional
\[
\frac{\rme^{-\tau}}{2} \int_0^1|\dot x|^2\rmd t+\rme^{\tau}\int_0^1U(x)\rmd t
\]
for real numbers $\tau$
and $1$-periodic curves $x=x(t)$ in $Q$.
The critical points are solutions of
$\ddot x=\rme^{2\tau}\grad_xU$
such that $\frac12|\dot x|^2=\rme^{2\tau}U(x)$,
see Section \ref{critpoint}.
Periodic solutions are in
one to one correspondence
with $T$-periodic solutions in the theorem.
We will prove that solutions exist
provided the conditions \eqref{inj} -- \eqref{ac2}
are satisfied.

In Section \ref{action} we describe
the variational problem.
In Section \ref{comp}
we study compactness properties
of the penalized action.
Notice,
that during the depenalization
process in Lemma \ref{boudbelow}
we use a comparison argument
in contrast to the asymptotic flow box
in \cite{bpv09,bprv13}.
This allows us to weaken the assumption
on the metric.
In Section \ref{mountain}
we prove existence of critical points
of the penalized problem.
We use a linking construction
as in \cite{bpv09,bprv13}
but in a different way.
A rearrangement of Morse handles
of the negative set
allows us to apply gradient
flow methods directly.
Therefore,
no requirements on the homology
of the loop space as in \cite{bprv13}
are necessary,
cf.\ Section \ref{reeb}.

\subsection{Reeb orbits\label{reeb}}

The kernel of the canonical symplectic form
$\rmd\lambda$ on $T^*Q$ restricted to
$TM$ defines a $1$-dimensional distribution
the so-called {\bf characteristic distribution}
on the energy surface $M$.
A closed leaf of the foliation
is called a {\bf closed characteristic}.
Because the characteristic foliation
is integrated by the Hamiltonian vector field
$X_H$ the theorem implies:

\begin{cor3}
  Assuming \eqref{inj} -- \eqref{ac2}
  the energy surface $M$ carries a closed characteristic.
\end{cor3}

We remark that the energy surface $M$
of the Hamiltonian $H$ is of restricted contact type
in $T^*Q$.
The Liouville form $\lambda$
in the direction of $X_H$
is twice the kinetic energy $|\,.\,|^2$.
This follows because $X_H$
is the sum of $-1$ times
the cogeodesic vector field
and the $\rmd\lambda$-dual of $-\pi^*V$,
cf.\ \cite{geig08}.
Therefore, if $M$ is disjoint from the zero section
the restriction of $\lambda$ to the characteristic
distribution is positive,
i.e. $M$ is of {\bf restricted contact type},
see \cite{hoze94,cie98}.
In the alternative case
it suffices to find a function
$f$ on $T^*Q$
such that $\lambda+\rmd f$ is positive
along the restriction of $X_H$
to $M$.
Because $X_H$ is nowhere tangent to $Q$
along $\partial N$
we can use flow box neighbourhoods
for $X_H$ on $M$
which cover $M\cap Q$,
a partition of unity,
and the fundamental theorem of calculus
to construct a function $h$
on $M$ such that
$\rmd h(X_H)>-\lambda(X_H)$.
An extension of $h$ to $T^*Q$
yields $f$ as desired.
In particular, the restriction $\alpha$
of $\lambda+\rmd f$
to $TM$ defines a contact form on $M$
together with its {\bf Reeb vector field} $R$
which is defined by
$i_R\rmd\alpha=0$
and $\alpha(R)=1$.
In particular,
$R$ is tangent to
the characteristic foliation.
The integral curves are called
{\bf Reeb orbits}.
Therefore, the corollary ensures
existence of closed Reeb orbits
in the present situation.
For more on closed Reeb orbits
the reader is referred to,
cf.\ \cite{hoze94,gin05,pas12}.

Notice,
that if a component of the negative
set $N$ has compact closure
the existence of a periodic orbit
follows from the classical
results mentioned above.
To see this replace $Q$
with the double of the manifold
with boundary $\bar{N}$.
On the complement of $\bar{N}$
define the metric and the potential suitably
such that the new negative
set coincides with $N$.
Alternatively,
on can glue $\partial N\times [0,\infty)$
to $N$ along the boundary
using a collar neighbourhood induced by
$\grad V/|\grad V|^2$.
The result can be provided with a metric
which interpolates to the product metric on
the cylindrical end.
Further the potential is the projection to the $\R$-factor on
$\partial N\times [0,\infty)$.
Then the classical existence results
follow from the theorem
because $(N,\partial N)$
carries a fundamental class with coefficients modulo $2$.

\begin{cor3}
  If the energy surface $M$
  is compact then there exists
  a closed characteristic on $M$
  which is contractible
  provided the dimension of
  $M$ is at least two.
\end{cor3}

\begin{rem}
  More generally,
  the gluing of 
  $\partial N\times [0,\infty)$
  to $N$
  as described above
  yields a Riemannian manifold $Q$
  with bounded geometry
  provided $\partial N$
  is compact.
  Consequently,
  the connected components of $M$
  for which the intersection
  with $Q$ is compact
  and
  whose image under $\pi$
  satisfy the homological condition
  \eqref{lnk}
  carry a (contractible)
  closed characteristic.
  Therefore, {\it a priori}
  we can assume that
  those components $N'$
  are replaced with
  $\partial N'\times (-\infty,0]$
  where the potential is interpolated
  to the projection to the
  $\R$-component in a
  bounded neighbourhood of
  $\partial N'$.
\end{rem}

We claim that the theorem implies
the main result of \cite{bprv13}.
This is of interest only
if all connected components of $\bar{N}$
are non-compact.
Assuming this we claim that
\[
H_k(N,\partial N)\cong
H_{k+n-1}(M)
\]
naturally
for $2\leq k\leq n$
provided $Q$ is orientable
or the coefficients are taken modulo $2$.
As Geiges pointed out to us
a proof can be given as follows:
Denote by $D'$ the unit codisc bundle
of $Q$ restricted to $\partial N$
and its boundary sphere bundle by $S'$.
Similarly, we denote the induce sphere
bundle over $N$
by $S$.
W.l.o.g.\
we can assume that $M$
equals $S\cup_{S'}D'$.
By the excision property of homology
we have that
$H_{k+n-1}(S,S')$
is isomorphic to
$H_{k+n-1}(M,\partial N)$.
Therefore,
it suffices to prove
$H_k(N,\partial N)\cong H_{k+n-1}(S,S')$,
which is obtained
with the
Gysin sequence for the bundle pair $(S,S')$
over $(N,\partial N)$,
see \cite[Proposition 12.1]{dld72}.
Because $\bar{N}$ is not compact
this holds for $k=1$ as well.
Therefore,
the above isomorphism is correct for $k=1$
provided $\partial N$
is not compact.
In the compact case
we observe that $H_n(M)$
injects into $H_n(M,\partial N)$
and hence into $H_1(N,\partial N)$.
In conclusion, if the homology
of $M$ is non-trivial
for some degree $*=n,\ldots,2n-1$
the condition \eqref{lnk} is satisfied.


\section{The Lagrangian action\label{action}}

\subsection{Admissible curves}{\label{admcurves}}

We identify the circle $S^1$ with the $1$-dimensional torus $\R/\Z$
and fix an isometric embedding of $Q$ into a Euclidean space $\R^N$
as it is possible by a theorem of Nash \cite[Theorem 3]{nash56},
cf.\ also \cite{gun89,gun91}.
The Hilbert manifold $H^1$ of absolutely continuous maps $S^1\ra Q$
with square integrable derivative
can be obtained as a submanifold of $H^1(S^1,\R^N)$,
cf.\ \cite{kli78}.
The tangent space $T_xH^1$ is spanned
by vector fields $\xi\in H^1(x^*TQ)$ along the loop $x\in H^1$.
The Riemannian metric $\langle\,.\,,\,.\,\rangle$ on $Q$
defines a Riemannian structure on $H^1$ via
\[
\langle\xi,\eta\rangle_1=
\int_0^1\langle\xi,\eta\rangle\rmd t+\int_0^1\langle\dot\xi,\dot\eta\rangle\rmd t,
\]
where we denote with $\dot\xi$
the covariant derivative $\nabla_{\!\dot x\,}\xi$
along the curve $x$.
The distance between $x$ and $y$ in $H^1$
equals the minimal length of curves
connecting $x$ and $y$.
This gives $H^1$ the structure of a complete metric space.
This follows with \cite[Theorem 1.4.5]{kli78}
where only the completeness of $Q$ is used.

Denote by $\MM\subset H^1$ the submanifold
which consists of contractible loops $x$ in $Q$.
Notice that $\MM$ is the connected component
of the isometrically embedded
and totally geodesic submanifold of point curves again denoted by
$Q$ in $H^1$.

\subsection{The parametrized Lagrangian action functional}{\label{paramlag}}

The {\bf energy}
\[
\EE(x)=\frac12 \int_0^1|\dot x|^2\rmd t
\]
and the {\bf potential integral}
\[
\UU(x)=\int_0^1U(x)\rmd t
\]
of a curve $x$ define smooth functions on $H^1$.
For real parameters $\tau$ the {\bf Lagrangian action}
is defined by
\[
\LLL(x,\tau)=\rme^{-\tau}\EE(x)+\rme^{\tau}\UU(x).
\]
Its restriction to $\MM$
is also denoted by $\LLL$.

\subsection{The penalty term}{\label{penlag}}

For $\varepsilon>0$ and
\[
P(\tau)=\rme^{-\tau}+\rme^{\tau/2}
\]
we add $\varepsilon P$ to the Lagrangian action to obtain
\[
\LLL_{\varepsilon}(x,\tau)=\LLL(x,\tau)+\varepsilon P(\tau).
\]

\subsection{Critical points}{\label{critpoint}}

Let $(x,\tau)$ be a critical point of $\LLL_{\varepsilon}$,
i.e. a point on which the linearization $T_{(x,\tau)}\LLL_{\varepsilon}$
vanishes.
Observe that
\[
T_{(x,\tau)}\LLL_{\varepsilon}(\xi,0)
=\rme^{-\tau}\int_0^1\langle\dot x,\dot\xi\rangle\rmd t
+\rme^{\tau}\int_0^1\langle\grad_xU,\xi\rangle\rmd t.
\]
With integration by parts we obtain a solution
of the {\bf Euler-Lagrange equation}
\[
\ddot x=\rme^{2\tau}\grad_xU,
\]
where $\ddot x=\nabla_{\!\dot x\,}\dot x$.
It follows from elliptic regularity that the solutions are smooth.

Taking
\[
T_{(x,\tau)}\LLL_{\varepsilon}(0,1)
=-\rme^{-\tau}\EE(x)
+\rme^{\tau}\UU(x)
+\varepsilon P'(\tau)
\]
we obtain the sum of the negative of the {\bf parametrized Hamiltonian action} and
($\varepsilon$ times) the derivative
\[
P'(\tau)=-\rme^{-\tau}+\frac12\rme^{\tau/2}.
\]
Because $(x,\tau)$ is a critical point the Euler-Lagrange equation
shows
that the {\bf parametrized Hamiltonian}
\[
\frac12\rme^{-\tau}|\dot x|^2-\rme^{\tau}U(x)=\varepsilon P'(\tau)
\]
is an integral of motion.
We refer to this also as the {\bf energy identity}.

Taking sum and difference of the critical value $c_{\varepsilon}=\LLL_{\varepsilon}(x,\tau)$
and $T_{(x,\tau)}\LLL_{\varepsilon}(0,1)$ we obtain
\[c_{\varepsilon}=2\rme^{\tau}\UU(x)+\varepsilon\frac32\rme^{\tau/2},\]
resp.,
\[c_{\varepsilon}=2\rme^{-\tau}\EE(x)+\varepsilon\left(2\rme^{-\tau}+\frac12\rme^{\tau/2}\right).\]


\section{Compactness\label{comp}}

\subsection{Palais-Smale property\label{ps}}

We consider a Palais-Smale sequence $(x_{\nu},\tau_{\nu})$ of $\LLL_{\varepsilon}$,
i.e. the sequence of linear operators $T_{(x_{\nu},\tau_{\nu})}\LLL_{\varepsilon}$
converges to zero and $\LLL_{\varepsilon}(x_{\nu},\tau_{\nu})$ converges to a
real number $c_{\varepsilon}$ as $\nu$ tends to infinity.

\begin{prop}
  \label{psprop}
  $(x_{\nu},\tau_{\nu})$ has a convergent subsequence.
\end{prop}

\begin{proof}
  The difference of $\LLL_{\varepsilon}(x_{\nu},\tau_{\nu})$
  and $T_{(x_{\nu},\tau_{\nu})}\LLL_{\varepsilon}(0,1)$,
  which can be estimated by
  $\varepsilon\left(2\rme^{-\tau}+\frac12\rme^{\tau/2}\right)\geq\frac32\varepsilon$
  from below,
  tends to $c_{\varepsilon}$.
  Therefore, the sequences $|\tau_{\nu}|$ and hence $\EE(x_{\nu})$ are bounded.
  We can assume that $\tau_{\nu}$ converges to $\tau_*$.

  We claim that
  \[
  \sup_{t\in S^1}\dist\big(o,x_{\nu}(t)\big)
  \]
  is bounded.
  We argue by contradiction.
  Because of the bound on the energy $\EE(x_{\nu})$
  we obtain a bound on the length of $x_{\nu}$.
  Therefore, we can assume that $x_{\nu}(S^1)$
  is contained in $Q\setminus\hat{Q}$.
  We consider the vector field
  \[
  \xi_{\nu}=\frac{\grad_{x_{\nu}}U}{|\grad_{x_{\nu}}U|^2}
  \]
  along $x_{\nu}$,
  which is well defined by \eqref{ac1a}.
  With $\Hess U=\langle\nabla\grad U,\,.\,\rangle$
  and \eqref{ac1b}
  we obtain that $\|\xi_{\nu}\|_1^2$
  is bounded by a positive constant times $1+\EE(x_{\nu})$. 
  Therefore,
  $T_{(x_{\nu},\tau_{\nu})}\LLL_{\varepsilon}(\xi_{\nu},0)$
  tends to zero
  because $(x_{\nu},\tau_{\nu})$ is a Palais-Smale sequence.
  But a direct computation using \eqref{ac1a} and \eqref{ac2} shows
  that the limit equals $\rme^{\tau_*}$.
  This is a contradiction.
  
  We claim that a subsequence of $(x_{\nu},\tau_{\nu})$
  converges in $C^0(S^1,Q)\times\R$.
  Observe that
  \[
  \dist\big(x_{\nu}(t_0),x_{\nu}(t_1)\big)
  \leq\length\big(x_{\nu}|_{t_0}^{t_1}\big)
  \leq\sqrt{|t_1-t_0|}\sqrt{2\EE(x_{\nu})}.
  \]
  The bound on $\EE(x_{\nu})$ shows that the sequence $x_{\nu}$
  is equicontinuous.
  Because the Riemannian manifold $Q$ is complete
  the theorem of Arzel\`a-Ascoli applies.
  
  We can assume that $x_{\nu}\ra x_*$ in $C^0$.
  Approximating $x_*$ by a smooth loop $y$ in $Q$
  we can further assume that the sequence $x_{\nu}$
  is contained in a chart of $\MM$ about $y$,
  w.l.o.g. $\MM=H^1(y^*TQ)$.
  Because $y$ is contractible we find an orthogonal
  trivialization of $y^*TQ$.
  For the following computations
  we further assume $Q=\R^n$ with the Euclidean metric
  by uniform equivalence
  of the Riemannian metrics
  on a compact set.
  Cf.\ also with \cite[p.~26-27]{kli78}.
  In other words
  we can assume that $\MM$ equals the Hilbert space
  $H^1(S^1,\R^n)$.
  
  Using Fourier series representations as in \cite{abb01,hoze94}
  the norm of an element $x$ of $H^1(S^1,\R^n)$
  can be estimated from above by
  $\|x\|_{\infty}^2+2\EE(x)$.
  In order to show that $x_{\nu}$
  is a Cauchy sequence it suffices to
  show that $\EE(x_{\nu}-x_{\mu})$ tends to zero
  for $\nu,\mu\ra\infty$.
  Because $(x_{\nu},\tau_{\nu})$ is a Palais-Smale sequence
  $T_{(x_{\nu},\tau_{\nu})}\LLL_{\varepsilon}(x_{\nu}-x_{\mu},0)\ra 0$.
  Hence, the integral $\int\langle\dot x_{\nu},\dot x_{\nu}-\dot x_{\mu}\rangle$
  equals $\rme^{2\tau_{\nu}}\int\langle\grad_{x_{\nu}}U,x_{\nu}-x_{\mu}\rangle$
  up to a term which tends to zero.
  Because $x_{\nu}$ converges in $C^0$ (and a symmetry argument)
  $\EE(x_{\nu}-x_{\mu})$ tends to zero as well.
  Hence, $x_{\nu}$ is a Cauchy sequence in $\MM$,
  which converges to $x_*$.
\end{proof}

\subsection{Depenalization\label{dp}}

As we will show in Section \ref{mountain}
there exist positive constants $K_1<K_2$
and a sequence $\varepsilon\searrow 0$
such that $\LLL_{\varepsilon}$ carries a critical point
$(x,\tau)=(x_{\varepsilon},\tau_{\varepsilon})$
whose critical value $c_{\varepsilon}$
is contained in the interval $[K_1,K_2]$.

\begin{lem}
  The sequence $\tau_{\varepsilon}$ is bounded above.
\end{lem}

\begin{proof}
  The sum
  \[
  c_{\varepsilon}=\LLL_{\varepsilon}(x,\tau)+T_{(x,\tau)}\LLL_{\varepsilon}(\xi,-1)
  \]
  equals the following expression
  \[
  \int_0^1
  \Big(
  \rme^{-\tau}
  \big(
  |\dot x|^2+\langle\dot x,\dot\xi\rangle
  \big)
  +
  \rme^{\tau}\langle\grad_xU,\xi\rangle
  \Big)
  \rmd t
  +
  \varepsilon\left(2\rme^{-\tau}+\frac12\rme^{\tau/2}\right).
  \]
  Removing the $\varepsilon$-term and plugging in
  \[
  \xi=\delta\frac{\grad_xU}{1+|\grad_xU|^2}
  \]
  for some $\delta>0$ yields
  \[
  c_{\varepsilon}>
  \int_0^1
  \Big(
  \rme^{-\tau}
  \big(
  |\dot x|^2+\langle\dot x,\dot\xi\rangle
  \big)
  +
  \rme^{\tau}\delta\frac{|\grad_xU|^2}{1+|\grad_xU|^2}
  \Big)
  \rmd t.
  \]
  Due to \eqref{ac1a} and \eqref{ac1b} $|\dot\xi|$
  is bounded by an $\varepsilon$-independent positive
  constant times $\delta|\dot x|$.
  We choose $\delta$ such that $|\dot\xi|\leq\frac12|\dot x|$.
  This implies $|\dot x|^2+\langle\dot x,\dot\xi\rangle\geq\frac12|\dot x|^2$
  and therefore,
  \[
  c_{\varepsilon}>\rme^{\tau}
  \int_0^1
  \Big(
  \rme^{-2\tau}\frac12|\dot x|^2
  +
  \delta\frac{|\grad_xU|^2}{1+|\grad_xU|^2}
  \Big)
  \rmd t.
  \]
  It suffices to bound the integrand $I_{\varepsilon}$
  from below.
  
  If the curves $x_{\varepsilon}$ stay outside
  the compact set $\hat{Q}$ a lower bound
  is given by $\delta (1+K^2)^{-1}$
  using \eqref{ac1a}.
  In the alternative case we find $t_{\varepsilon}\in S^1$
  such that $x_{\varepsilon}(t_{\varepsilon})\in\hat{Q}$.
  We can assume that $\tau_{\varepsilon}\geq0$
  because otherwise there is nothing to show.
  This implies that the multiplied energy identity
  \[
  \frac12\rme^{-2\tau_{\varepsilon}}|\dot x_{\varepsilon}|^2
  -U(x_{\varepsilon})
  =\varepsilon\left(\frac12\rme^{-\tau_{\varepsilon}/2}-\rme^{-2\tau_{\varepsilon}}\right)
  \lra 0
  \]
  tends to zero
  independently of $t$
  because it takes values in the interval
  $(-\frac{\varepsilon}{2},\frac{\varepsilon}{2})$.
  Consider the set $T_{\varepsilon}$ of $t\in S^1$
  for which the first term of the multiplied energy identity
  is bounded away from zero by a small constant.
  The set $T_{\varepsilon}$ is measurable.
  The integrand $I_{\varepsilon}$ restricted to $T_{\varepsilon}$
  is bounded below as desired.
  On the complement $S^1\setminus T_{\varepsilon}$
  we can assume that
  $U(x_{\varepsilon})$
  is uniformly close to zero.
  In other words, $x_{\varepsilon}\big(S^1\setminus T_{\varepsilon}\big)$
  is contained in a small neighbourhood of $\partial N\cap\hat{Q}$
  on which $|\grad U|$ is uniformly positive.
  This implies a lower bound of $I_{\varepsilon}$
  on the complement as well.
\end{proof}

\begin{lem}
  \label{boudbelow}
  The sequence $\tau_{\varepsilon}$ is bounded below.
\end{lem}

\begin{proof}
  Arguing by contradiction
  we assume that $\tau_{\varepsilon}\ra-\infty$
  as $\varepsilon$ tends to zero.
  Because $K_2>c_{\varepsilon}>2\rme^{-\tau_{\varepsilon}}\EE(x_{\varepsilon})$,
  see Section \ref{critpoint},
  we infer
  $
  \EE(x_{\varepsilon})\ra 0
  $.
  Therefore, \[\length(x_{\varepsilon})\ra 0.\]
  Again with Section \ref{critpoint} we get
  $
  K_1
  <c_{\varepsilon}
  =2\rme^{\tau_{\varepsilon}}\UU(x_{\varepsilon})
  +\varepsilon\frac32\rme^{\tau_{\varepsilon}/2}.
  $
  Hence,
  the sequence $\UU(x_{\varepsilon})$ is unbounded.
  This implies
  \[
  \inf_{t\in S^1}\dist\big(o,x_{\varepsilon}(t)\big)\lra\infty.
  \]
  Consequently,
  we can assume that $x_{\varepsilon}(S^1)$
  is contained in
  the intersection of $Q\setminus\hat{Q}$
  and the geodesic ball $B(q)$ of radius $\varrho<\frac12\inj g$
  about a point $q$ on the curve $x_{\varepsilon}$.
  Moreover, with \eqref{ac1a} the solution $x_{\varepsilon}$
  of the Euler-Lagrange equation
  is not constant.
  The following arguments will lead to a contradiction.
  
  With help of the Euler-Lagrange equation for $x_{\varepsilon}$
  we observe
  \[
  \frac{\rmd^2}{\rmd t^2}U\big(x_{\varepsilon}(t)\big)
  =\Big(\Hess_{x_{\varepsilon}(t)}U\Big)\big(\dot x_{\varepsilon}(t),\dot x_{\varepsilon}(t)\big)
  +\rme^{2\tau_{\varepsilon}}|\grad_{x_{\varepsilon}(t)}U|^2.
  \]
  For the maximum $t_{\varepsilon}$
  of the function $t\mapsto U\big(x_{\varepsilon}(t)\big)$
  on the circle this yields
  \[
  \rme^{2\tau_{\varepsilon}}|\grad_{x_{\varepsilon}(t_{\varepsilon})}U|^2
  \leq\big|\Hess_{x_{\varepsilon}(t_{\varepsilon})}U\big|\;|\dot x_{\varepsilon}(t_{\varepsilon})|^2.
  \]
  Invoking \eqref{ac1b} the estimate implies
  \[
  \rme^{2\tau_{\varepsilon}}|\grad_{x_{\varepsilon}(t_{\varepsilon})}U|
  \leq K|\dot x_{\varepsilon}(t_{\varepsilon})|^2.
  \]
  Moreover, the multiplied energy identity
  \[
  \frac12|\dot x_{\varepsilon}|^2
  -\rme^{2\tau_{\varepsilon}}U(x_{\varepsilon})
  =\varepsilon\left(\frac12\rme^{3\tau_{\varepsilon}/2}-1\right)
  \lra 0
  \]
  plugged in gives
  \[
  \rme^{2\tau_{\varepsilon}}|\grad_{x_{\varepsilon}(t_{\varepsilon})}U|
  \leq K\Big(o(1)+\rme^{2\tau_{\varepsilon}}U\big(x_{\varepsilon}(t_{\varepsilon})\big)\Big).
  \]
  In order to show that
  the right hand side tends to zero
  we use the following {\bf mean value argument}.
  By Section \ref{critpoint} there exists $t_0$ such that
  \[
  \rme^{2\tau_{\varepsilon}}U\big(x_{\varepsilon}(t_0)\big)
  <\frac12K_2\rme^{\tau_{\varepsilon}}.
  \]
  We denote by $c$ the unit speed geodesic
  from $c(0)=x_{\varepsilon}(t_0)$
  to $c(s_0)=x_{\varepsilon}(t_{\varepsilon})$
  inside the geodesic ball $B\big(x_{\varepsilon}(t_{\varepsilon})\big)$.
  With the fundamental theorem of calculus
  we obtain
  $U\big(x_{\varepsilon}(t_{\varepsilon})\big)\leq U\big(x_{\varepsilon}(t_0)\big)+\int_0^{s_0}|\grad_{c(s)}U|\rmd s$.
  An application of Gr\"onwall's lemma
  to the function $s\mapsto |\grad_{c(s_0-s)}U|$ gives
  $|\grad_{c(s)}U|\leq e^{Ks_0}|\grad_{x_{\varepsilon}(t_{\varepsilon})}U|$
  for all $s\in [0,s_0]$
  using \eqref{ac1a} and \eqref{ac1b}.
  Therefore,
  \[
  \rme^{2\tau_{\varepsilon}}|\grad_{x_{\varepsilon}(t_{\varepsilon})}U|
  < K
  \Big(
  o(1)
  +\frac12K_2\rme^{\tau_{\varepsilon}}
  +s_0\rme^{2\tau_{\varepsilon}+Ks_0}|\grad_{x_{\varepsilon}(t_{\varepsilon})}U|\Big).
  \]
  Because the length of $x_{\varepsilon}$
  (and hence $s_0$)
  tends to zero
  we can choose $\varepsilon$ such that  
  \[
  \rme^{2\tau_{\varepsilon}}|\grad_{x_{\varepsilon}(t_{\varepsilon})}U|
  <K
  \Big(
  o(1)
  +K_2\rme^{\tau_{\varepsilon}}\Big).
  \]
  Invoking Gr\"onwall's lemma
  along geodesics connecting $x_{\varepsilon}(t_{\varepsilon})$
  with the boundary of $B\big(x_{\varepsilon}(t_{\varepsilon})\big)$
  we obtain as above  
  \[
  \rme^{2\tau_{\varepsilon}}|\grad U|
  \lra 0
  \]
  uniformly on $B\big(x_{\varepsilon}(t_{\varepsilon})\big)$.

  \begin{rem}
    \label{speedesti}
    Consider the geodesic ball $B^{\varepsilon}$
    of radius $\frac12\varrho$ with center
    $x_{\varepsilon}(t_0)$.
    If $\varepsilon>0$ is sufficiently small
    $B^{\varepsilon}$ is contained in
    $B\big(x_{\varepsilon}(t_{\varepsilon})\big)$
    and contains $x_{\varepsilon}(S^1)$.
    Integrating along geodesics
    which start at $x_{\varepsilon}(t_0)$
    shows
    \[
    U\big(x_{\varepsilon}(t)\big)
    \leq U\big(x_{\varepsilon}(t_0)\big)
    +\frac14\inj g\;\,\sup_{B^{\varepsilon}}|\grad U|
    \]
    for all $t\in S^1$.
    Hence,
    \[
    \rme^{2\tau_{\varepsilon}}U(x_{\varepsilon})
    \leq\frac12K_2\rme^{\tau_{\varepsilon}}
    +o(1)
    \]
    uniformly on $S^1$.
    By the multiplied energy identity
    $
    \frac12|\dot x_{\varepsilon}|^2\leq o(1)
    $
    tends uniformly to zero.
    Therefore,
    we can assume that $|\dot x_{\varepsilon}|^2<1$
    on $S^1$.
  \end{rem}

  We continue the proof of the lemma.
  With \eqref{ac2} we obtain the stronger estimate
  \[
  \rme^{2\tau_{\varepsilon}}|\grad_{x_{\varepsilon}(t_{\varepsilon})}U|
  \leq o(1)\,|\dot x_{\varepsilon}(t_{\varepsilon})|^2
  \]
  as $\varepsilon$ tends to zero.
  The aim is to find a similar estimate
  for all $t\in S^1$
  with a variation of the above mean value argument.
  We consider a unit speed geodesic $c$
  inside $B\big(x_{\varepsilon}(t)\big)$
  connecting $c(0)=x_{\varepsilon}(t)$
  with $c(s_0)=x_{\varepsilon}(t_{\varepsilon})$.
  With the fundamental theorem of calculus
  and the Gr\"onwall's lemma applied
  to the function $s\mapsto |\grad_{c(s)}U|$
  we obtain
  \[
  U\big(x_{\varepsilon}(t_{\varepsilon})\big)-U\big(x_{\varepsilon}(t)\big)
  \leq
  s_0\rme^{Ks_0}|\grad_{x_{\varepsilon}(t)}U|
  \]
  using \eqref{ac1a} and \eqref{ac1b}.
  Combining this with the difference
  of the multiplied energy identities
  gives
  \[
  |\dot x_{\varepsilon}(t_{\varepsilon})|^2-|\dot x_{\varepsilon}(t)|^2
  \leq
  2s_0\rme^{2\tau_{\varepsilon}+Ks_0}|\grad_{x_{\varepsilon}(t)}U|.
  \]
  We will use this to estimate
  $\rme^{2\tau_{\varepsilon}}|\grad_{x_{\varepsilon}(t)}U|$
  from above.
  The Gr\"onwall's lemma gives a bound
  by $\rme^{2\tau_{\varepsilon}+Ks_0}|\grad_{x_{\varepsilon}(t_{\varepsilon})}U|$,
  which with the initial estimate is bounded by
  $o(1)\,\rme^{Ks_0}|\dot x_{\varepsilon}(t_{\varepsilon})|^2$.
  Therefore,
  \[
  \rme^{2\tau_{\varepsilon}}|\grad_{x_{\varepsilon}(t)}U|\leq
  o(1)
  \Big(
  |\dot x_{\varepsilon}(t)|^2+
  2s_0\rme^{2\tau_{\varepsilon}+Ks_0}|\grad_{x_{\varepsilon}(t)}U|
  \Big).
  \]
  Choosing $\varepsilon$ (and hence $s_0$ and $|\dot x_{\varepsilon}(t)|^2$,
  see Remark \ref{speedesti})
  sufficiently small we obtain the linear estimate
  \[
  \rme^{2\tau_{\varepsilon}}|\grad_{x_{\varepsilon}(t)}U|
  \leq o(1)\,|\dot x_{\varepsilon}(t)|
  \]
  uniformly for all $t\in S^1$.
  
  The desired contradiction will be achieved
  with the following comparison argument:
  Choose $\varrho$ smaller or equal than
  the harmonic injectivity radius.
  We assume that $Q=\R^n$
  using harmonic coordinates
  on $B(q)$ for a point $q$ on the curve $x_{\varepsilon}$.
  By \cite[Theorem 1.2]{hebey99} and \eqref{cb}
  the metric $\langle\,.\,,\,.\,\rangle$
  is uniformly equivalent to the Euclidean
  metric $\langle\,.\,,\,.\,\rangle_0$
  and the Christoffel symbols $\Gamma$
  are uniformly bounded.
  In particular, with Remark \ref{speedesti}
  \[
  \Big|\Gamma_{x_{\varepsilon}(t)}\big(\dot x_{\varepsilon}(t),\dot x_{\varepsilon}(t)\big)\Big|_0
  \leq o(1)\,|\dot x_{\varepsilon}(t)|_0
  \]
  uniformly in $t\in S^1$.
  Consequently,
  we have using the Euler-Lagrange equation
  \[
  \big|\dot x_{\varepsilon}(t_1)-\dot x_{\varepsilon}(t_2)\big|_0
  \leq\int_{t_1}^{t_2}\Big|\frac{\rmd}{\rmd t}\dot x_{\varepsilon}(t)\Big|_0\rmd t
  \leq o(1)\int_0^1|\dot x_{\varepsilon}(t)|_0\rmd t,
  \]
  for all $t_1,t_2\in[0,1]$.
  Let $t_1$ be the maximum
  of $t\mapsto |\dot x_{\varepsilon}(t)|_0$.
  Let $t_2$ be a point
  such that $\dot x_{\varepsilon}(t_2)$ vanishes
  or is perpendicular to $\dot x_{\varepsilon}(t_1)$
  w.r.t.\ the Euclidean metric.
  With the Pythagorean theorem
  \[
  |\dot x_{\varepsilon}(t_1)|_0\leq
  \big|\dot x_{\varepsilon}(t_1)-\dot x_{\varepsilon}(t_2)\big|_0\leq
  o(1)\,|\dot x_{\varepsilon}(t_1)|_0.
  \]
  This is a contradiction
  because the curve $x_{\varepsilon}$
  is not constant.
\end{proof}

The above lemmata ensure bounds on the
sequence $\tau_{\varepsilon}$ of Lagrangian multipliers.
A repetition of the arguments from
Proposition \ref{psprop} proves:

\begin{prop}
  \label{confore0}
  The sequence of critical points $(x_{\varepsilon},\tau_{\varepsilon})$
  has a convergent subsequence as $\varepsilon$ tends to zero.
\end{prop}

In particular the limit curve is a critical point of $\LLL$
with vanishing parameterized Hamiltonian energy.


\section{Mountain pass\label{mountain}}

The aim of this section is to prove the following existence statement,
which in view of Section \ref{dp} and Proposition \ref{confore0} proves Theorem \ref{mthm}:

\begin{prop}
  \label{existence}
  There exist positive constants $K_1$ and $K_2$
  with $K_1<K_2$
  such that for all $\varepsilon\in(0,K_1)$
  there exist $\varepsilon_0\in(0,\varepsilon)$
  and a critical point $(x,\tau)$ of $\LLL_{\varepsilon_0}$
  such that
  $
  K_1\leq\LLL_{\varepsilon_0}(x,\tau)\leq K_2.
  $
\end{prop}

\subsection{Begin of the proof\label{begproof}}

Arguing by contradiction we find a sequence
$\varepsilon_{\nu}$ of positive real numbers
such that for all $\varepsilon\in(0,\varepsilon_{\nu})$
the interval $[1/\nu,\nu]$
contains no critical value of $\LLL_{\varepsilon}$.
We will lead this assumption
to a contradiction in several steps
organized as separate sections.

\subsection{A deformation of the negative set\label{defneg}}

By \eqref{ac1a} the negative potential function $U$
has no critical point in the complement
of the compact set $\hat{Q}$.
Hence, in view of \eqref{reg}
the level sets $\{U=\pm\delta\}$
are isotopic to $\{U=0\}$
for $\delta>0$ sufficiently small.
An isotopy is given by following
the (negative) gradient flow lines of $U$.
We set 
\[
N_{\pm\delta}=\{U>\pm\delta\}.
\]
Notice that $N=N_0$.

\begin{lem}
  \label{distisotop}
  There exist $\delta>0$
  and an open subset $Q_B\subset Q$
  such that the pairs
  $(Q_B,\partial Q_B)$ and $(N_{\delta},\partial N_{\delta})$
  are isotopic and
  the minimal distance
  $\dist(\partial N_{\delta},Q_B)$
  is positive.
\end{lem}

\begin{proof}
  Consider the function
  \[
  f=\frac{U}{\sqrt{1+|\grad U|^2}}.
  \]
  The set $Q_B$ is defined by
  \[
  Q_B=\{f>\sqrt{2}\delta\}\subset N_{\delta}.
  \]
  Notice that $\partial N=\{f=0\}$.
  Invoking \eqref{ac1a} and \eqref{ac1b}
  there exists $\delta'>0$,
  which only depends on $K$,
  such that $|\grad f|$ is uniformly positive
  on $\{|f|<\delta'\}$.
  Because the metric on $Q$ is complete
  we can assume by shrinking $\delta'>0$
  that there exists
  a complete vector field $X$ on $Q$
  which coincides with $|\grad f|^{-2}\grad f$
  on $\{|f|<\delta'\}$.
  The flow of $X$
  brings $N$ to $Q_B$
  provided we choose $\delta<\delta'/\sqrt{2}$.
  This yields the desired isotopy.

  In order to show positivity of
  $\dist(\partial N_{\delta},Q_B)$
  consider a point $q$ in $Q_B$.
  Notice that
  \[
  U(q)>\delta\big(1+|\grad_qU|\big).
  \]
  Choose $r\in(0,\inj g)$
  and consider the geodesic ball $B_r(q)$ of radius $r$ about $q$.
  We assume that $B_r(q)$ is contained in $Q\setminus\hat{Q}$.
  Each point $p$ on the boundary $\partial B_r(q)$
  can be connected with the center $q$
  by a radial unit speed geodesic.
  By a mean value argument
  analogously to the application of Gr\"onwall's lemma
  in Lemma \ref{boudbelow}
  we obtain
  \[
  U(q)-U(p)\leq r\rme^{Kr}|\grad_qU|
  \]
  using \eqref{ac1a} and \eqref{ac1b}.
  Combining both estimates yields
  \[
  U(p)>\delta+|\grad_qU|(\delta-r\rme^{Kr}).
  \]
  Choose $r$ such that $r\rme^{Kr}\leq\delta$.
  Hence, $U(p)>\delta$.
  In other words $U>\delta$ on any
  geodesic ball of radius $r$
  in $Q\setminus\hat{Q}$ about points in $Q_B$.
  Because $\hat{Q}$ is compact
  the shortest length of a curve connecting
  points of $\partial N_{\delta}$
  with those of $Q_B$ is positive.
\end{proof}

\subsection{The linked set\label{bbbbbb}}

We define a subset
\[
\BB=\BB'\times\R
\]
of $\MM\times\R$ via
\[
\BB'=\big\{x\in\MM\,|\,\EE(x)=r\;\;\text{and}\;\;x(0)\in Q_B\big\}.
\]
Observe that the length of the curves $x\in\BB'$ is bounded by $\sqrt{2r}$.
In view of Lemma \ref{distisotop} we choose $r>0$ such that
\[
\length(x)\leq\sqrt{2r}<2\dist(\partial N_{\delta},Q_B).
\]
Therefore, $x(S^1)\subset N_{\delta}$.
In other words $\UU(x)\geq\delta$
so that the restriction of $\LLL$ to $\BB$
is bounded from below by $\rme^{-\tau}r+\rme^{\tau}\delta$.
Consequently,
\[
\LLL\geq 2\sqrt{r\delta}
\qquad\text{on}\quad\BB.
\]

\begin{rem}
  \label{forconv}
  We shrink $r>0$ further
  such that the energy functional $\EE$
  has no critical points on
  $\{0<\EE\leq r\}$.
  In view of the positivity assumption
  on the injectivity radius this is not a restriction.  
\end{rem}

\subsection{The linking set\label{aaaaaa}}

By \eqref{lnk} and Section \ref{defneg}
there is a simplicial cycle $c$
in $N_{-\delta}$ relative $\partial N_{-\delta}$
which is non-trivial in homology,
cf.\ \cite{munk63}.
We identify $Q$ with its image in $\MM$.
Observe,
\[
\LLL=-\rme^{\tau}\delta
\qquad\text{on}\quad|\partial c|\times\{\tau\},
\]
where we denote with $|\partial c|$
the geometric realization of the simplicial cycle
$\partial c$
(the boundary taken of the absolute chain $c$)
as a subset of 
$\partial N_{-\delta}$
and $\tau$ is a real number.
For $\tau_0\ll -1$ we find
\[
\LLL\leq\rme^{\tau_0}\max_{|c|}U
\qquad\text{on}\quad|c|\times\{\tau_0\}.
\]
The following lemma will be proved in Section \ref{thoftnp}.

\begin{lem}
  \label{constrofcs}
  There exists a chain $c'$ in $\MM$
  homotopic to $c$
  with boundary fixed
  such that $\UU(x)\leq -\frac12\delta$
  for all $x\in|c'|$.
\end{lem}

\begin{rem}
  \label{infwcnonempty}
  Notice that all curves $x$ in the chain $|c'|$
  leave $N_{\delta}$
  if $x(0)$ is in the closure of $Q_B$.
  This is because the negative potential integral $\UU$
  of $x$ is negative by Lemma \ref{infwcnonempty}.
  Therefore,
  \[
  r<\inf\{\EE(x)\,|\,x\in|c'|\;\;\text{and}\;\;x(0)\in Q_B\}
  \]
  by the choice of $r$ in Section \ref{bbbbbb}.
\end{rem}

By compactness of $|c'|$
the energy is bounded on $|c'|$.
Therefore,
we find $\tau_1\gg 1$
such that
\[
\LLL<0
\qquad\text{on}\quad|c'|\times\{\tau_1\}
\]
uniformly.
Let $\AAA$ be the union
\[
\AAA=
\Big(|\partial c|\times [\tau_0,\tau_1]\Big)\cup
\Big(|c|\times\{\tau_0\}\Big)\cup
\Big(|c'|\times\{\tau_1\}\Big)
\]
so that we obtain
\[
\sup_{\AAA}\LLL\leq\rme^{\tau_0}\max_{|c|}U.
\]
In view of the assumption in Section \ref{begproof}
we choose $\nu\in\N$
such that $1/\nu<\sqrt{r\delta}$.
We choose $\tau_0\ll -1$ such that
\[
\sup_{\AAA}\LLL<\frac{1}{\nu}.
\]
This implies
\[
\sup_{\AAA}\LLL<\inf_{\BB}\LLL.
\]
In particular $\AAA$ and $\BB$
are disjoint.

\subsection{A chain\label{cccccc}}

Define
\[
\CC=
\Big(|c|\times [\tau_0,0]\Big)\cup
\bigcup_{s\in[0,1]}\Big(|c_s|\times\{0\}\Big)\cup
\Big(|c'|\times [0,\tau_1]\Big)
\]
where $c_s$ is the homotopy from $c_0=c$ to $c_1=c'$
in $\MM$ relative $\partial c$
which we will construct in Section \ref{thoftnp},
cf.\ Lemma \ref{constrofcs}.
By construction $\AAA$ and $\CC$
can be given the structure of simplicial chains
such that $\partial\CC=\AAA$.
In particular $\AAA$ is a cycle.
By compactness of $\CC$ we can assume
that additionally
\[
\sup_{\CC}\LLL<\nu.
\]
Increasing $\nu$
amounts to decreasing $\tau_0$.
But this does not effect the above estimates.

\subsection{The minmax argument\label{minmax}}

By compactness of $\AAA$ and $\CC$
we find $\varepsilon\in(0,\varepsilon_{\nu})$
such that
\[
\sup_{\AAA}\LLL_{\varepsilon}<1/\nu<\inf_{\BB}\LLL_{\varepsilon}
\]
and
\[
\sup_{\CC}\LLL_{\varepsilon}<\nu.
\]

\begin{lem}
  \label{prodwind}
  The action window set
  $\LLL_{\varepsilon}^{-1}\big([1/\nu,\nu]\big)$
  is diffeomorphic to
  \[\{\LLL_{\varepsilon}=1/\nu\}\times[0,\nu-1/\nu].\]
\end{lem}

\begin{proof}
  Notice that by Section \ref{begproof} there are no
  critical points in
  $\WW=\LLL_{\varepsilon}^{-1}\big([1/\nu,\nu]\big)$.
  We define a vector field
  \[
  \xi=\frac{\grad\LLL_{\varepsilon}}{|\grad\LLL_{\varepsilon}|_1^2}
  \]
  on $\WW$ and consider its flow, cf.\ \cite{abb01,lang99}.
  By the Palais-Smale property in Section \ref{ps}
  the vector field
  $\xi$ is of bounded length
  $\sup_{\WW}|\xi|_1<\infty$.
  We can assume that $\xi$
  is extended to $\MM\times\R$
  via a partition of unity
  such that $\xi$ has support
  in a slightly larger action window.
  By completeness of $\MM\times\R$
  the flow $\varphi$ of $\xi$
  is global.
  The desired diffeomorphism is
  \[
  \big((x,\tau),s\big)\longmapsto\varphi_s(x,\tau)
  \]
  for
  $(x,\tau)\in\{\LLL_{\varepsilon}=1/\nu\}$
  and
  $s\in [0,\nu-1/\nu]$.
\end{proof}

Therefore,
the relative cycle $\CC$
in $\big(\{\LLL_{\varepsilon}\leq\nu\},\{\LLL_{\varepsilon}<1/\nu\}\big)$
is homologically trivial.
Notice,
that by
Remark \ref{infwcnonempty}
the intersection
$\CC\cap\BB$
is non-empty
and that
$\AAA=\partial\CC$ and
$\BB$ are disjoint.
Moreover,
$\BB\subset\MM\times\R$ is a hypersurface
defined via a smooth function,
see Remark \ref{forconv}.
Similarly,
$\partial\BB$ is contained
in the preimage of $\partial Q_B$
under the surjective submersion
\[
\begin{array}{rcl}
  \ev:\MM\times\R & \lra & Q \\
  (x,\tau) & \longmapsto & x(0)
\end{array}
\]
and $\partial Q_B$
is the zero set of a smooth function.
By a generic {\it a posteriori}
choice of $r$ and $\delta$
we can assume that $\CC\cap\BB$
defines a cycle in $(\BB,\partial\BB)$,
which is trivial by the above discussion.

Let $\DD$ be the intersection of $\CC$
with
\[
\big\{
(x,\tau)
\,|\,
\EE(x)\leq r
\;\;\text{and}\;\;
x(0)\in Q_B
\big\}.
\]
$\DD$ has the structure of a simplicial
chain with boundary in the union of
\[
\BB
\;\;\text{with}\;\;
\ev^{-1}(\partial Q_B)
\;\;\text{and with}\;\;
|c|\times\{0\}.
\]
Therefore,
the cycles $c\cap\bar{Q}_B$
and $\ev(\CC\cap\BB)$
are homologous in $(Q_B,\partial Q_B)$
via $\ev(\DD)$.
The triviality of $\CC\cap\BB$
implies the triviality of $\ev(\CC\cap\BB)$
so that the cycle $c\cap\bar{Q}_B$ is
trivial in relative homology
of $(Q_B,\partial Q_B)$.
Because $(N,\partial N)$ and $(Q_B,\partial Q_B)$ are isotopic
the cycle $c$ is trivial in $(N,\partial N)$.
This contradicts the choice of $c$.

\subsection{Handles of the negative potential\label{thoftnp}}

In order to finish the proof
of Proposition \ref{existence}
we prove Lemma \ref{constrofcs}.

\begin{lem}
  \label{construtilde}
  There exists a smooth function $\tilde{U}$ on $Q$
  and a compact subset $\hat{N}\subset N$
  such that
  \begin{itemize}
  \item $\tilde{U}\geq U$ and $\tilde{U}=U$ on $Q\setminus\hat{N}$ and
  \item the restriction of $\tilde{U}$ to $N$ is a Morse function
    without local minima.
  \end{itemize}
\end{lem}

\begin{proof}
  By a local perturbation of $U$ we find a
  function $\tilde{U}$
  as in the lemma,
  see \cite[Section 2]{mil65},
  but eventually with positive local minima.
  With \eqref{ac1a}
  we find a vector field $X$
  on $Q$ which
  equals
  \[\frac{\grad U}{|\grad U|^2}\]
  on $Q\setminus\hat{Q}$,
  has bounded length
  (is therefore complete),
  and is gradient-like for $\tilde{U}$ on $N$,
  \cite[Lemma 3.2]{mil65}.
  The aim is to remove all local minima of $\tilde{U}$ by
  a cancellation process as described in \cite{mil65}.
  
  With
  \cite[Lemma 2.8]{mil65}
  we can assume that different critical points
  in $N$ have different critical values.
  For regular values $0\leq a<b$
  the manifold with boundary
  \[
  W_{\!ab}=\tilde{U}^{-1}\big([a,b]\big)
  \]
  is called an
  {\bf action window set}.
  $W_{\!ab}$ is called {\bf regular}
  if $\tilde{U}$ has no critical point
  in $W_{\!ab}$.
  Using the flow of $X$ as in Lemma \ref{prodwind}
  $W_{\!ab}$ is diffeomorphic to
  $\{\tilde{U}=a\}\times[0,b-a]$,
  cf.\ \cite[Theorem 3.4]{mil65}.
  $W_{\!ab}$ is called {\bf elementary}
  if $\tilde{U}$ has exactly one critical point $q_0$
  on $W_{\!ab}$.
  The flow lines of $X$
  whose closure does not intersect $q_0$
  connect $\{\tilde{U}=a\}$
  with $\{\tilde{U}=b\}$.
  The intersection $S_L(q_0)$
  of flow lines of $X$ with $\{\tilde{U}=a\}$
  which connect with $q_0$
  in forward time
  is diffeomorphic to a sphere
  of dimension $\ind(q_0)-1$,
  where $\ind(q_0)$ denotes the
  Morse index of $q_0$.
  The intersection $S_R(q_0)$
  with $\{\tilde{U}=b\}$
  in backward time
  is diffeomorphic to a sphere
  of dimension $n-\ind(q_0)-1$.
  We call $S_L(q_0)$ the
  {\bf left-hand sphere} of $q_0$
  and $S_R(q_0)$ the {\bf right-hand sphere}
  of $q_0$,
  see \cite[Definition 3.9]{mil65}.
  
  In order to alter $\tilde{U}$
  into a self-indexing-like
  Morse function on $N$
  we consider the composition
  \[
  W_{\!ac}=W_{\!ab}\cup W_{\!bc}
  \]
  of elementary action window sets.
  The critical points are denoted by
  $q_0\in W_{\!ab}$
  and
  $q_1\in W_{\!bc}$.
  If $\ind(q_1)\leq\ind(q_0)$
  a compactly supported diffeotopy of
  $\{\tilde{U}=b\}$
  yields a gradient-like
  vector field of $\tilde{U}$
  which coincides with
  $X$ near $\partial W_{\!ac}$
  and outside a compact set
  such that the right- and left-hand
  spheres $S_R(q_0)$ and $S_L(q_1)$
  in $\{\tilde{U}=b\}$ are disjoint,
  cf.\ \cite[Theorem 4.4]{mil65}.
  Therefore,
  the compact sets $K(q_0)$,
  resp., $K(q_1)$,
  of flow lines of the vector field
  (again denoted by) $X$
  connecting $q_0$,
  resp., $q_1$,
  in $W_{\!ac}$ are disjoint.
  As in \cite[Theorem 4.1]{mil65}
  we can increase the function
  $\tilde{U}$ in a neighbourhood
  of $K(q_0)$ keeping the critical
  points $q_0$ and $q_1$
  such that $X$ is still gradient-like
  and the critical value of $q_0$
  lies above the critical value of $q_1$.
  Moreover,
  near $\partial W_{\!ac}$ the Morse
  function is not changed.
  In other words,
  after a rearrangement of the critical points
  on $N$ we obtain a function $\tilde{U}$
  on $Q$
  which coincides with $U$
  on $Q\setminus\hat{N}$ for a compact subset
  $\hat{N}$ of $N$
  such that $\tilde{U}\geq U$
  and for all positive critical points
  $q_0$ and $q_1$ of $\tilde{U}$
  we have
  \begin{itemize}
  \item if $\ind(q_0)=\ind(q_1)$ then $\tilde{U}(q_0)=\tilde{U}(q_1)$,
  \item if $\ind(q_0)<\ind(q_1)$ then $\tilde{U}(q_0)<\tilde{U}(q_1)$.
  \end{itemize}
  I.e. $\tilde{U}$ behaves like a self-indexing Morse function on $N$,
  see \cite[Theorem 4.8]{mil65}.
  
  If $\tilde{U}$ has no positive local minimum
  we are done.
  It remains to consider
  the alternative case.
  With \eqref{ac1a} and the flow of $-X$
  Courants minmax argument
  as in \cite[Theorem 4.2]{str96}
  applies to the set of paths $\gamma$
  connecting a positive local minimum $\gamma(0)$
  with a point $\gamma(1)$ outside $N$.
  Therefore, there exists a positive saddle point
  of Morse index $1$.
  Because $\tilde{U}$ is
  self-indexing-like on $N$
  there exists an index $1$ 
  positive saddle point $q_1$
  which is connected with a
  positive local minimum $q_0$
  via exactly one flow line $T$ of $X$.
  
  Increasing $\tilde{U}(q_0)$
  and $\tilde{U}(q)$
  slightly for all index $1$
  positive saddle points $q\neq q_1$
  we can assume that $q_0$ and $q_1$
  are the critical points
  of the composition $W_{\!ac}$
  of the elementary action window sets
  $W_{\!ab}$ and $W_{\!bc}$.
  We claim that the first cancellation theorem
  \cite[Theorem 5.4]{mil65}
  applies:
  Let $D_R(q_0)$
  be the $n$-dimensional
  {\bf right-hand disc} of $q_0$
  which by definition is the union
  of all flow lines of $X$
  in $W_{\!ab}$ starting at $q_0$.
  Let $K_T$ be the compact neighbourhood
  of $T$ which is the union of
  $D_R(q_0)\cup T$
  with the set of flow lines of $X$
  in $W_{\!bc}$
  which ends in a small compact
  tubular neighbourhood of
  the hypersurface $S_R(q_1)$
  in $\{\tilde{U}=c\}$.
  Notice, that if a flow line in $W_{\!ac}$
  leaves $K_T$ once it never comes back.
  Following the arguments in
  \cite[p.~51ff]{mil65}
  we can alter the vector field $X$
  inside a neighbourhood of $T$ in $K_T$
  such that the flow of
  the vector field
  (again denoted by)
  $X$ yields a diffeomorphism
  from $W_{\!ac}$ to $\{\tilde{U}=a\}\times[0,c-a]$.
  With the construction in
  \cite[p.~54]{mil65}
  and $X(\tilde{U})=1$ on $Q\setminus\hat{Q}$
  there exists a new function
  $\tilde{U}$
  which coincides with the old one
  outside a compact set
  and on $\partial W_{\!ac}$
  such that $X$ is gradient-like
  for $\tilde{U}$.
  Because $\tilde{U}$
  increases along the flow of $X$
  the new function has no critical point in
  $W_{\!ac}$.
  Further, it can be assumed to be
  greater or equal than the old.
  
  Repeating this argument we can
  remove all positive local minima.
  This proves the lemma.
\end{proof}

\begin{proof}[{\bf Proof of Lemma \ref{constrofcs}}]
  Let $Y$ be a vector field on $Q$
  which does not vanish in a compact neighbourhood of $|c|$.
  As in the proof of Lemma \ref{construtilde}
  we consider a complete gradient-like vector field $X$
  for $\tilde{U}$.
  We can assume that the spaces of flow lines $\NN$ of $X$
  connecting positive critical points of $\tilde{U}$
  are manifolds of dimension $\leq n-1$,
  cf.\ \cite{sch93}.
  We can perturb $Y$ not to be tangent
  to $\NN$
  at the points of $|c|$.
  
  Following the flow of $Y$
  on a small interval around $0$
  we find a chain $\tilde{c}$ in $\MM$
  which is homotopic to $c$
  relative $\partial c$.
  The chain $\tilde{c}$ is obtained
  from $c$ by adding small loops induced by $Y$
  which start at points on $|c|$.
  We can assume that
  the loops starting on $|\partial c|$
  are constant;
  those starting on $|c|\cap\{U>-\delta/2\}$ not.
  Additionally,
  the intersections of every loop
  with $\NN$
  are uniformly finite.
  
  Let $Z$ be a complete vector field on $Q$
  which coincides with $-X$ on $N$
  and vanishes on $\{U\leq-\delta\}$.
  Applying the flow of $Z$
  to the loops representing $\tilde{c}$
  we get a $1$-parameter family $c_s$ in $\MM$
  starting at $c_0=\tilde{c}$
  with boundary $\partial c_s$ fixed.
  Moreover,
  all $x\in|c_s|$
  converge to arcs in $\{U\leq-\delta/2\}$
  in $C^{\infty}_{\loc}$
  outside the intersections
  with $\NN$.
  Therefore,
  we have
  $
  \UU(x)\leq\int\tilde{U}(x)\rmd t\leq-\delta/2
  $
  for all $x\in|c_s|$
  and $s$ sufficiently large.
\end{proof}


\begin{ack}
  First and foremost we would like to thank the Mathematische
  Forschungsinstitut Oberwolfach and its
  research in pairs program for its hospitality and the stimulating research
  environment.
  Second we like to thank
  Hansj\"org Geiges for his inspiring course on celestial mechanics in the
  winter term 2012/13 at the
  Universit\"at zu K\"oln
  and for providing us with the homological argument
  in Section \ref{reeb}.
  At the same time we like to thank
  Victor Bangert for pointing our attention towards the results
  on brake orbit solutions of mechanical Lagrangian systems.
  In addition, we would like to thank Alexander Lytchak
  for suggesting the use of harmonic coordinates.
  Further we would like to thank Alberto Abbondandolo, Peter Albers,
  Barney Bramham, and Federica Pasquotto for helpful discussions on
  this subject as well as Olaf M\"uller
  and Marc Nardmann for explaining us their result in \cite{muenar13}
  on conformal geometry.
\end{ack}


\end{document}